\documentclass[11pt,a4paper]{article}

\usepackage{doi}

\usepackage{graphicx,amssymb,amsmath,amsthm}
\usepackage{xspace}
\usepackage{hyperref}
\usepackage{xcolor}

\usepackage{fullpage}

\newtheorem{theorem}{Theorem}
\newtheorem{proposition}[theorem]{Proposition}
\newtheorem{lemma}[theorem]{Lemma}

\DeclareMathOperator{\diam}{\mathsf{diam}}
\DeclareMathOperator{\radius}{\mathsf{radius}}
\DeclareMathOperator{\len}{\mathsf{len}}
\DeclareMathOperator{\cyc}{\mathsf{cyc}}
\DeclareMathOperator{\leaves}{\ell}

\newcommand{\eps}{\varepsilon}

\let\leq\leqslant
\let\geq\geqslant
\let\le\leqslant
\let\ge\geqslant

\makeatletter
\partopsep\z@ \textfloatsep 10pt plus 1pt minus 4pt
\def\section{\@startsection {section}{1}{\z@}%
  {-3.5ex plus -1ex
    minus -.2ex}{2.3ex plus .2ex}{\large\bf}}
\def\subsection{\@startsection{subsection}{2}%
  {\z@}{-3.25ex plus
    -1ex minus -.2ex}{1.5ex plus .2ex}{\normalsize\bf}}
\def\@fnsymbol#1{\ensuremath{\ifcase#1\or 1\or 2\or 3\or 4\or
    5\or 6\or 7 \or 8\or 9 \or 10\or 11 \else\@ctrerr\fi}}
\makeatother

\title{Diameter and Length of Metric Graphs}

\author{Hee-Kap Ahn%
  \thanks{Graduate School of Artificial Intelligence, Department of
    Computer Science and Engineering, Pohang University of Science and
    Technology. heekap@postech.ac.kr} 
  \and
  Sergio Cabello%
  \thanks{Faculty of Mathematics and Physics, University of Ljubljana, Slovenia, 
    and Institute of Mathematics, Physics and Mechanics, Slovenia. 
    sergio.cabello@fmf.uni-lj.si. 
    Funded in part by the Slovenian Research and Innovation Agency
    (P1-0297, N1-0218, N1-0285, J1-70045) and in part by the European
    Union (ERC, KARST, project number 101071836). Views and opinions
    expressed are however those of the authors only and do not
    necessarily reflect those of the European Union or the European
    Research Council. Neither the European Union nor the granting
    authority can be held responsible for them.} 
  \and
  Otfried Cheong%
  \thanks{SCALGO, Aarhus, Denmark. otfried@scalgo.com}
  \and
  Delia Garijo%
  \thanks{Department of Applied Mathematics I and Institute of
    Mathematics, University of Seville,
    Spain. dgarijo@us.es. Partially supported by grants
    PID2023-150725NB-I00, RED2024-153572-T and SOL2024-31596.} 
  \and
  Jeongwon Moon%
  \thanks{Department of Computer Science and Engineering, Pohang
    University of Science and Technology. jwmoon@postech.ac.kr} 
}

\date{}

\begin{document}

\maketitle

\begin{abstract}
  A metric graph is a metric space obtained from a finite collection
  of intervals whose endpoints are identified in groups.  It can also
  be seen as a finite, edge-weighted graph where the continuum of
  points along the interior of each edge is taken into consideration,
  and each edge is locally isometric to an interval whose length is
  the edge-weight. The diameter of a metric graph $G$ is the maximum
  distance between all pairs of points of $G$.

  We show that the total length of a metric graph $G$ with
  $\leaves(G)$ leaves, cyclomatic number $\cyc(G)$, and diameter
  $\diam(G)$ is at most $\big(\cyc(G) + \max\big\{1, \leaves(G)/2
  \big\}\big) \cdot \diam(G)$. Furthermore, we show that this bound is
  tight, and we characterize the metric graphs where equality holds.
  As an application, we provide tight bounds in certain cases for the
  diameter of metric graphs obtained from a cycle or a star by the
  identification of a fixed number of points (pairwise or in groups).
\end{abstract}


\section{Introduction}
\label{sec:intro}

A {\it metric graph} is a metric space obtained from a finite
collection of intervals whose endpoints are identified in groups.  A
metric graph can also be seen as a finite, edge-weighted graph
$(V,E,w\colon E\rightarrow \mathbb{R})$ where the continuum of points
along the interior of each edge is taken into consideration, and each
edge $e\in E$ is locally isometric to an interval of
length~$w(e)$. With our definition,\footnote[6]{The theory of metric
  graphs naturally carries over to infinite collections of intervals
  or, alternatively, infinite graphs.  In the infinite case the space
  is not compact. We will restrict our attention to finite collections
  of edges.} a metric graph is a compact space. We will restrict our
attention to connected metric graphs.

Several concepts from classical graph theory naturally apply to metric
graphs, and we use them when no confusion can arise. For example, we
may say that a metric graph is a {\it path}, a {\it tree} or a {\it
  cycle}. Similarly, a point $x$ in a metric graph $G$ is a {\it leaf}
of $G$ if it is a leaf vertex in the combinatorial graph defining~$G$.
Alternatively, we could say that $x\in G$ is a leaf if there exists an
$\eps$-neighborhood of~$x$ in $G$ that is homeomorphic to an interval
with endpoint~$x$.

For any two points~$x,y$ in a metric graph~$G$, their {\it distance},
denoted by $d_G(x,y)$, is the minimum length of all paths in~$G$ that
connect~$x$ to~$y$.  A {\it shortest path from~$x$ to~$y$} is a path
of length $d_G(x,y)$ connecting $x$~to~$y$.  Note that the shortest
path between two points on the same edge is not necessarily contained
in the edge.

To continue the discussion, we introduce the following key parameters
associated with a metric graph $G$:
\begin{itemize}
\item The {\it length}~$\len(G)$ of~$G$ is the sum of the lengths of its edges. More  
  generally, the length of a (metric) subgraph is the sum of the lengths of 
  its edges.
\item The {\it diameter}~$\diam(G)$ of~$G$ is the maximum distance between any two points of $G$.
  That is, $\diam(G)=\max_{x\in G}\max_{y\in G} d_G(x,y)$.
\item The {\it radius}~$\radius(G)$ of~$G$ is $\min_{x\in G}\max_{y\in G} d_G(x,y)$.
\item The number~$\leaves(G)$ of leaves of~$G$.
\item The {\it cyclomatic number} $\cyc(G)$ of $G$, assuming that $G$ is connected, 
  is $|E(\mathcal{G})|-|V(\mathcal{G})|+1$, where $\mathcal{G}$ is a combinatorial graph 
  describing $G$. This is also the first Betti number of $G$ as a topological space, 
  and thus it is a topological invariant: different combinatorial graphs  
  describing the ``same'' metric graph give the same cyclomatic number. 
  It is also equivalent to the dimension of the cycle space.
\end{itemize}

\paragraph{Our main results.}

Our goal in this work is to understand the relationship between radius
or diameter on one side and the length of the metric graph on the
other side, taking into account~$\cyc(G)$ and~$\leaves(G)$. More
precisely, we show that for any metric graph $G$
\begin{align*}
  \text{Theorem~\ref{thm:len-radius}:}~~~ &\len(G) ~\leq~
  \big(2\cyc(G) + \leaves(G)\big) \cdot \radius(G),\\ 
  \text{Theorem~\ref{thm:len-diam}:}~~~ &\len(G) ~\leq~ \Big(\cyc(G) +
  \max\big\{1, \frac{\leaves(G)}{2}\big\}\Big) \cdot \diam(G). 
\end{align*}
Moreover, in both cases we characterize the metric graphs where
equality holds in these bounds.

Both bounds look very natural and we were surprised that we could not
find them in the literature. Technically, the proof of the first bound
is not difficult, but it does provide context for the proof of the
second bound, especially to explain why the simpler technique does not
work. To prove the second bound, we use induction on the cyclomatic
number and use a new operation, which we call {\it isometric path
  identification}. At a high level, for the inductive step, we find
and isometrically identify two paths of the same length to reduce the
cyclomatic number by one and to reduce the length of the metric graph
by at least $\diam(G)$.

It is important to note that the results hold for metric graphs,
but they do not hold for classical, combinatorial graphs,
that is, if the radius or the diameter are defined by
considering only the distances between vertices.
See Figure~\ref{fig:discrete} for examples.

\begin{figure}[htb]
  \centerline{\includegraphics[scale=1.1]{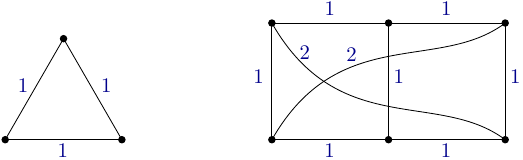}}
  \caption{Left: a graph with discrete radius $1$
    (considering only the distance between vertices),
    length $3$ and cyclomatic number $1$. As a metric graph,
    the radius is $3/2$.
    Right: A graph with discrete diameter $2$ (again considering
    only the distance between vertices), length $11$ 
    and cyclomatic number $4$.
    As a metric graph, the diameter is $5/2$.}
  \label{fig:discrete}
\end{figure}

\paragraph{Wormholes and fusion.}

Wormholes and fusion are two operations on metric graphs where some
points are identified. In the case of \emph{wormholes}, we identify
pairs of points, whereas in the case of \emph{fusion}, we identify a
set of points to a single point. These operations can also be
interpreted as adding zero-length edges to the metric graph.  Note
that, for points $x_1,\dots,x_k$ in a metric graph~$G$, the addition
of the wormholes $\{x_1,x_2\},\dots ,\{x_{k-1},x_{k}\}$ is the same as
the fusion of the set $\{x_1,\dots, x_k\}$. Thus, in general, adding
$k$ wormholes is at least as powerful as the fusion of $k+1$~points.

The following two questions naturally arise: 
\begin{itemize} 
\item For a metric graph $G$ and a budget $k$, can we reduce its
  diameter by adding $k$ wormholes, and how much can we reduce it?
\item For a metric graph $G$ and a budget $k$, can we reduce its
  diameter through a fusion of $k+1$ vertices, and how much can we
  reduce it?
\end{itemize}

We consider these problems in Section~\ref{sec:applications}.  Our
bounds on the total length of a metric graph are used to show some
optimal bounds for paths and stars.  The connection is that each time
we add a wormhole, the cyclomatic number increases by one, but the
total length remains constant.  We also provide some additional
results for stars using different techniques.

\paragraph{Context and related work.}

Metric graphs have been extensively considered in analysis; see for
example~\cite{zbMATH05317314,friedlander2005genericity,BFb0086338,KostenkoN2023}.
Much focus has been put on metric graphs equipped with differential
operators, which are usually called quantum graphs; see for example
the books~\cite{berkolaikoK2013,Kurasov2023} for a comprehensive
treatment.  They have also been studied for several optimization
problems, often under the name of ``continuous graphs'' or some
variant thereof. See for
example~\cite{CabelloR10,Dearing1974,GrigorievHLW21,hakimi1964optimum,HartmannLW22,Shier1977,Tamir1991}
for facility-location and packing problems, or~\cite{FGHHM24} for the
problem of finding a shortest tour that visits the
$\delta$-neighborhood of each point in a metric graph.

The diameter of metric graphs was studied by Chen and
Garfinkel~\cite{ChenGar82}---they developed a quadratic-time algorithm
(in the number of edges of the graph) to compute this parameter. In
this direction, in~\cite{CabelloGKKPS25}, the authors present
subquadratic-time algorithms to compute the diameter of some classes
of sparse metric graphs.

The mean distance of a metric graph $G$, defined as
$\frac{1}{\len(G)^2}\iint_{x,y\in G} d_G(x,y) \,dx \,dy$, has been
studied recently in~\cite{BKM24,GarijoMS2023}.  Particularly relevant
is the work of Baptista, Kennedy, and Mugnolo~\cite{BKM24} since they
characterize metric graphs with the \emph{minimum mean distance}.  In
contrast, we characterize the graphs with minimum diameter, that is,
minimum maximum distance.

We are not aware of any work considering wormholes in
\emph{combinatorial} or \emph{classical} graphs.  However, wormholes
are closely related to the addition of shortcuts to decrease the
diameter.  The use of wormholes is a neat abstraction in settings
where the cost of using new connections is negligible compared to the
cost of using the existing network.  The problem of adding wormholes
is completely intrinsic to the metric graph, while adding shortcuts,
that is, connections of some non-zero length between points of the
metric graph, requires ``ambient'' knowledge to determine the length
of such a connection.

Optimal placement of shortcuts to decrease the diameter has been
studied when the input objects lie in a Euclidean
space~\cite{BAE201937,CaceresGHMPR18,CGSS-17,CarufelMS16,EomASA26,GarijoMRS19,GrosseKSGS19}.
In this setting, the length of the shortcut is given by the Euclidean
length of the new chord, a property extrinsic to the input
object. When considering the diameter after the augmentation, one has
to take into account the new points lying on the shortcut, as they may
define the new diameter.

Shortcuts between vertices in classical graphs to minimize the
diameter have been considered
extensively~\cite{AGR2000,BiloGP12,DemaineZ10,DodisK99,FratiGGM15,GaoHN13,LiMS92,SchooneBL87}.
In some cases, all edges have unit weights, while in other cases
different weights have been considered.  In this setting, only
shortcuts connecting vertices are allowed and the diameter is
restricted to the vertices of the graph.  In our setting, we consider
connecting (with cost zero) arbitrary points and consider the
continuous version of the diameter (in the metric graph).  Both
settings present their own challenges and results from one do not seem
to easily carry to the other.

Computational aspects of decreasing the diameter of a combinatorial
graph using the fusion of vertices is introduced and studied
in~\cite{COMAS20091612}. They show that the problem is computationally
hard, but it can be solved in polynomial time for trees.  Compared to
this work, our fusion in metric graphs allows to merge arbitrary
points on the metric graph, instead of just vertices, and we consider
the continuous diameter as the measure to minimize. We can think of
our fusion as the limit process of the discrete setting, if we
subdivide each edge multiple times.

A natural attempt to show an upper bound on the length of a metric
graph $G$ is to argue that it can be covered by $k(G)$ shortest paths,
for some function $k(G)$. From this, we obtain that $\len(G)\le
k(G)\cdot \diam(G)$.  This line of thought is closely related to the
concept of \emph{geodesic cover} (also known as isometric path cover),
which has been studied for combinatorial
graphs~\cite{Fitzpatrick99,ManuelKPA25,PanC06}. Particularly relevant
is the work by Chakraborty, Foucaud, and
Hakanen~\cite{ChakrabortyFH25} as they also use the cyclomatic number
as a main parameter.  To prove Theorem~\ref{thm:len-diam} using this
approach, we would need to argue that a metric graph can be covered
with $\cyc(G)+\max\{1,\leaves(G)/2\}$ shortest paths.  However, as the
example in Figure~\ref{fig:no-cover} shows, this is not true.  We will
encounter this example as a special case of the so-called \emph{Medusa
graph} in Section~\ref{sec:diameter}.
\begin{figure}[htb]
  \centerline{\includegraphics[width=.4\textwidth,page=2]{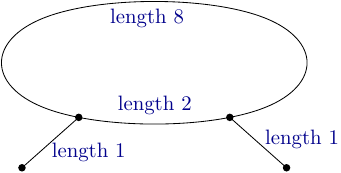}}
  \caption{A metric graph~$G$ with total length $12$ and diameter $6$,
    attained by the distance between a leaf and a point on the upper edge.
    It has cyclomatic number $1$ and two leaves, but cannot 
    be covered with $\cyc(G)+\max\{1,\leaves(G)/2\}=2$ shortest paths.}
  \label{fig:no-cover}
\end{figure}

\paragraph{Assumptions.}

As mentioned above, we will assume that the metric graphs are
connected, since otherwise our bounds on their total length are
trivial.  We will often assume that they have no vertices of degree
two; such vertices are not needed in metric graphs because we can
insert them by splitting edges and remove them by merging edges
without ``changing'' the metric graph.  This is particularly useful
when talking about the neighbor vertex of a leaf.


\section{Radius and length of metric graphs}
\label{sec:radius}

As a warm-up exercise, we prove the following theorem that relates the
radius and the length of a metric graph.  The proof will immediately
allow us to characterize the metric graphs for which the bound holds
with equality.
\begin{theorem}
  \label{thm:len-radius}
  For a metric graph~$G$ we have
  \[
  \len(G) \leq \big(2\cyc(G) + \leaves(G)\big) \cdot \radius(G).
  \]
\end{theorem}
\begin{proof}
  Let~$R := \radius(G)$ and let~$u$ be a center of~$G$, that is, a
  point such that~$d_G(u, x) \leq R$ for all~$x \in G$.  Let~$T$ be a
  shortest-path tree in~$G$ with origin~$u$. We observe that~$\len(T) =
  \len(G)$. 
  Every leaf of~$T$ is either a leaf of~$G$ or a point~$x$
  that has more than one shortest path from~$u$.  We can go back
  from~$T$ to~$G$ by identifying these points, each such
  identification increases the cyclomatic number by one, and removes
  one or two leaves.  It follows that~$\leaves(T) \le \leaves(G) +
  2\cyc(G)$.  Since~$T$ is entirely covered by the paths from~$u$ to
  all leaves, and each such path has length at most $R$,
  we have $\len(G) = \len(T)\le \leaves(T)\cdot R
  \le (\leaves(G) + 2\cyc(G))\cdot R$. 
\end{proof}

For which graphs~$G$ does the bound in Theorem~\ref{thm:len-radius}
hold with equality? From the proof it follows that this is the case
exactly when the shortest-path tree~$T$ from~$u$ has~$2\cyc(G) +
\leaves(G)$ leaves, the paths from~$u$ to each leaf are pairwise
disjoint and each of them has length $R$.  This is the case exactly
if~$T$ is a star with $2\cyc(G) + \leaves(G)$ edges of length~$R$.  To
reconstruct~$G$ from~$T$, we must identify $\cyc(G)$ pairs of points
to obtain a graph with cyclomatic number $\cyc(G)$. Moreover, we have
to identify disjoint pairs of leaves of $T$. Indeed, any other
$\cyc(G)$ identifications using some point that is not a leaf of $T$
or identifying three or more leaves together gives a graph that has at
least $\leaves(G)+1$ leaves.  See Figure~\ref{fig:radius-equality} for
an example.  Thus, the graph~$G$ consists of a single non-leaf
vertex~$u$, $\leaves(G)$ edges of length~$R$ incident to~$u$, and
$\cyc(G)$ loops of length~$2R$ also incident to~$u$.  See the
rightmost graph in Figure~\ref{fig:radius-equality}.  In the next
section we will meet this graph again as an $R$-Medusa graph.

\begin{figure}[tb]
  \centerline{\includegraphics[]{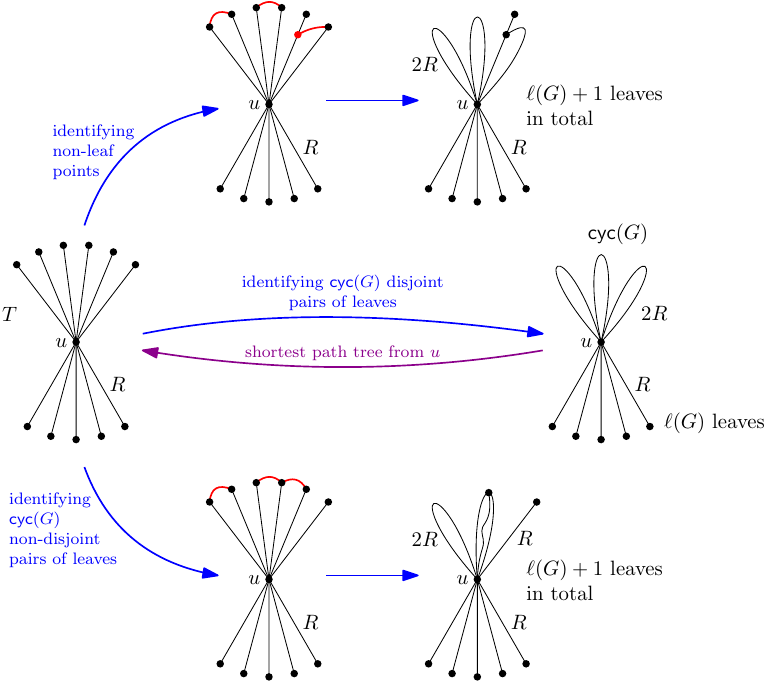}}
  \caption{Characterizing the longest metric graphs with fixed radius
    and number of leaves. In this example $\cyc(G)=3$ and
    $\leaves(G)=5$, and thus tree $T$ (on the left side) has $2\cdot
    3+5=11$ leaves.  Only the identification in the middle path gives
    a graph with cyclomatic number $3$, $5$ leaves and whose
    shortest-path tree gives $T$. The identifications in the upper or
    lower path give too many leaves.}
  \label{fig:radius-equality}
\end{figure}


\section{Diameter and length of metric graphs}
\label{sec:diameter}

In this section we turn to the diameter of the graph. Our main theorem relates the diameter
and the total length of a metric graph~$G$.  It is similar to Theorem~\ref{thm:len-radius}, 
but proving it requires considerably more effort.
\begin{theorem}
  \label{thm:len-diam}
  For a metric graph~$G$ we have
  \[
  \len(G) \leq \Big(\cyc(G) + \max\big\{1, \frac{\leaves(G)}{2}\big\}\Big) \cdot \diam(G).
  \]
\end{theorem}

The bound is again tight: we will characterize the graphs that achieve
it below in Theorem~\ref{thm:medusa}.  Before we turn to the general
theorem, let us treat the case~$\cyc(G)=0$, that is, when~$G$ is a
tree:
\begin{lemma}
  \label{lem:len-diam-tree}
  For a metric tree~$T$ we have
  \[
  \len(T) \leq \frac{\leaves(T)}{2} \cdot \diam(T).
  \]
\end{lemma}
\begin{proof}
  For one proof, use Theorem~\ref{thm:len-radius} and observe that in
  metric trees~$T$, $\diam(T) = 2 \radius(T)$. (For general metric spaces
  $\diam(G) \leq 2 \radius(G)$, but it is easy to see that for metric 
  trees the equality holds.)

  Another proof uses double-counting.  First, we sum up the lengths of
  the paths connecting every pair of leaves.  This sum is at most
  ${\ell \choose 2} \diam(T)$, where $\ell = \leaves(T)$. Next,
  consider an edge~$e$.  It appears in this summation at least~$\ell -
  1$ times, since it separates at least one leaf from the remaining
  leaves.  Summing over all edges, we obtain
  \[
  (\ell - 1) \cdot \len(T) \leq {\ell \choose 2} \cdot \diam(T)
  = \frac{\ell(\ell - 1)}{2}\cdot \diam(T),
  \]
  implying the lemma.
\end{proof}

For which metric trees~$T$ does the bound in Lemma~\ref{lem:len-diam-tree}
hold with equality?  Since $\diam(T) = 2 \radius(T)$ for metric trees, 
these are exactly the trees for which Theorem~\ref{thm:len-radius} holds with equality,
that is, the stars with edge length~$\diam(T)/2$.

Our main tool for proving Theorem~\ref{thm:len-diam} in general is
\emph{isometric path identification} in metric graphs.  Let~$G$ be a
metric graph, and let~$\alpha$ and~$\beta$ be two simple paths of the
same length~$\lambda$ in~$G$ that are disjoint, except possibly at
their endpoints.  We can consider~$\alpha$ and~$\beta$ as
parameterized curves $[0, \lambda] \mapsto G$ of unit speed.  We now
define a new metric graph~$G'$ by \emph{identifying} the
points~$\alpha(t)$ and~$\beta(t)$ for all~$0 \leq t \leq \lambda$. We
let~$\pi$ be the natural projection from~$G$ to~$G'$. See
Figure~\ref{fig:path-identification} for an illustration.
\begin{figure}[t]
  \centerline{\includegraphics[width=.9\textwidth]{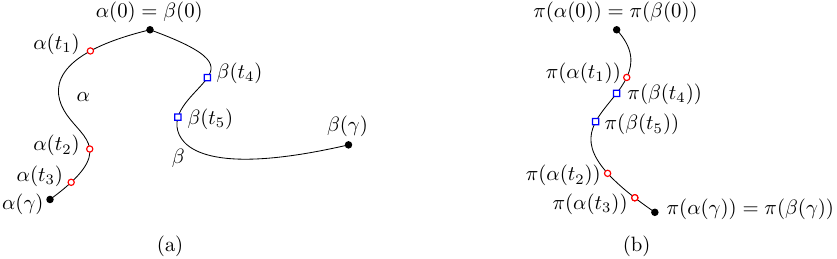}}
  \caption{(a) Simple paths $\alpha$ and $\beta$ of length
    $\lambda$. (b) Identification of $\alpha(t)$ and $\beta(t)$ for
    all $0\le t\le\lambda$.} 
  \label{fig:path-identification}
\end{figure}

\begin{lemma}
  \label{lem:iso-identi}
  Isometric path identification has the following properties:
  \begin{enumerate}
  \item[(i)] for all $x, y \in G$, we have $d_{G'}(\pi(x), \pi(y)) \leq
    d_{G}(x, y)$;
  \item[(ii)] $\diam(G') \leq \diam(G)$;
  \item[(iii)] $\len(G') = \len(G) - \lambda$;
  \item[(iv)] if $\alpha(0) = \beta(0)$ and $\alpha(\lambda) \neq
    \beta(\lambda)$, then~$\cyc(G') = \cyc(G)$;
  \item[(v)] if $\alpha(0) = \beta(0)$ and $\alpha(\lambda) =
    \beta(\lambda)$, then~$\cyc(G') = \cyc(G) - 1$.
  \end{enumerate}
\end{lemma}
\begin{proof}
  Consider a shortest path~$\gamma$ connecting~$x$ and~$y$ in~$G$. The
  projection~$\pi(\gamma)$ is a path of equal length (possibly it is
  not a simple path or not a shortest path), implying claim~(i).
  Claim~(ii) then follows immediately, and claim~(iii) is obvious.

  To prove the other two claims, we count edges and vertices of~$G$
  and~$G'$. Before doing so, we introduce new degree-two vertices
  on~$\alpha$ and~$\beta$ at all points that will be projected to a
  vertex of~$G'$ (in other words, at all locations on~$\alpha$ that
  correspond to a vertex on~$\beta$ and vice versa).  After this step,
  $\alpha$~and~$\beta$ consist of the same number~$e$ of
  edges. Let~$n$ and~$m$ be the number of vertices and edges of~$G$,
  such that~$\cyc(G) = m - n + 1$.  During path identification, $e$
  pairs of edges are identified, so~$G'$ has $m - e$ edges.  If
  $\alpha(0) = \beta(0)$ and $\alpha(\lambda) \neq \beta(\lambda)$,
  then $e$ pairs of vertices are identified, so $G'$ has $n - e$
  vertices and therefore~$\cyc(G') = (m - e) - (n - e) + 1 = m - n +
  1$, proving claim~(iv).  If $\alpha(0) = \beta(0)$ and
  $\alpha(\lambda) = \beta(\lambda)$, then only the $e-1$ inner path
  vertices are identified, so~$G'$ has $n - (e - 1)$ vertices, and
  $\cyc(G') = (m - e) - (n - e + 1) + 1 = m - n = \cyc(G) - 1$,
  proving claim~(v).
\end{proof}

We will use isometric path identification in the form described in the
following lemma.

\begin{lemma}
  \label{lem:identi-cycle}
  Let $G$ be a connected metric graph containing a cycle~$C$. If $G
  \neq C$, then there is a metric graph~$G'$ 
  such that~$\diam(G') \leq \diam(G)$,
  $\len(G') = \len(G) - \len(C)/2$, $\cyc(G') = \cyc(G) - 1$, and
  $\leaves(G') \leq \leaves(G) + 1$.
\end{lemma}
\begin{proof}
  Let $u$ be a vertex of~$C$ of degree larger than two (it exists
  since $G \neq C$ and $G$ is connected), and let~$v$ be the
  \emph{antipode} of~$u$ on~$C$, that is, the point (not necessarily a
  vertex) on~$C$ such that the two ``half-cycles'' $C[u,v]$ and
  $C[v,u]$ have length~$\len(C)/2$.  We obtain~$G'$ by isometric path
  identification of these two half-cycles of length~$\len(C)/2$
  connecting~$u$ and~$v$.  Leaves of~$G'$ are either leaves of~$G$ or
  the point~$v$; the degree of points in the interior of the
  half-cycles cannot decrease, and the degree of $u$ decreases exactly
  by one, and thus is at least $2$. The other claims follow directly
  from Lemma~\ref{lem:iso-identi}.
\end{proof}

For a leaf~$v$ of~$G$, let us call the other endpoint of the single
edge incident to~$v$ the \emph{neighbor}~$n(v)$, and let us
call~$d_G(v, n(v))$ the \emph{length}~$\lambda(v)$ of the leaf~$v$.
\begin{lemma}
  \label{lem:leaf-equal}
  From a metric graph~$G$, we can obtain a graph~$G'$ with
  $\leaves(G') = \leaves(G)$, $\cyc(G') = \cyc(G)$, $\diam(G') \leq
  \diam(G)$, $\len(G') \geq \len(G)$, and in~$G'$ all leaves have
  equal length and share the same neighbor.
\end{lemma}
\begin{proof}
  If $G$ contains two leaves~$v_1$ and~$v_2$ with~$n(v_1) \neq
  n(v_2)$, then assume that~$\lambda(v_1) \ge \lambda(v_2)$.  We can
  simply remove the leaf~$v_2$ from~$n(v_2)$ and reattach it
  at~$n(v_1)$. This operation does not change the length, the cyclomatic
  number or the number of leaves. It may change the diameter, but it
  does not increase it. We repeat this process until all leaves 
  have the same neighbor.
 
  Let now~$v_{1}, \dots, v_{\ell}$ be the~$\ell = \leaves(G)$ leaves
  ordered such that~$\lambda(v_1) \geq \lambda(v_2) \geq \dots \geq
  \lambda(v_\ell)$. 
  Set $\lambda = \frac 12 (\lambda(v_1) +
  \lambda(v_2))$ and note that $\lambda(v_i)\le \lambda$ for all $3\le i\le \ell$.  
  We can now change all leaves to have
  length~$\lambda$: this cannot increase the diameter and can only
  increase the length. The resulting graph is $G'$.
\end{proof}

We now have the tools we need to tackle Theorem~\ref{thm:len-diam}.
\begin{proof}[Proof of Theorem~\ref{thm:len-diam}]
  We proceed by induction over~$\cyc(G)$. For~$\cyc(G) = 0$, the
  theorem holds by Lemma~\ref{lem:len-diam-tree}, so we assume that we
  are given a graph~$G$ with~$k := \cyc(G) > 0$ and that the claimed
  bound holds for all graphs with cyclomatic number smaller
  than~$k$. We let~$D := \diam(G)$ and distinguish two cases.

  If~$\leaves(G) < 2$, we let~$C$ be a shortest cycle in~$G$.  This
  implies that~$\len(C) \leq 2D$. We apply Lemma~\ref{lem:identi-cycle}
  with the cycle~$C$.  The resulting graph~$G'$ has $\cyc(G') = k - 1$
  and $\leaves(G') \leq \leaves(G) + 1 \leq 2$, so by the inductive 
  assumption we have
  \[
  \len(G') \leq \big( \cyc(G') + \max\{1, \leaves(G')/2\}\big) \cdot \diam(G')
  = k \cdot \diam(G').
  \]
  It follows that
  \[
  \len(G) = \len(G') + \len(C)/2 \leq k \cdot \diam(G') + D
  \leq k \cdot D + D = (k + 1) \cdot D,
  \]
  completing the proof for this case.

  Otherwise, $\leaves(G) \geq 2$.  By Lemma~\ref{lem:leaf-equal}, we
  can assume that all leaves~$v_1, \dots, v_{\ell}$, where~$\ell =
  \leaves(G)$, have equal length~$\lambda$ and share the same
  neighbor~$x := n(v_i)$.
  Let $G^\circ$ be the part of~$G$ without
  the leaves and let~$z$ be a point of~$G^\circ$ maximizing the
  distance from~$x$; see Figure~\ref{fig:lambda-medusa}(a). Since~$G^\circ$ 
  has no leaves, $z$ is not a leaf. Take a sufficiently small $\eps>0$ 
  and let $\pi_\eps(z)$ be a path in $G$ of length $\eps$ with $z$ 
  in its interior. Shortest paths from $x$ to points in $\pi_\eps(z)$ 
  do not go through $z$ because $z$ maximizes the distance from $x$. 
  It follows that there must be at least 
  two distinct shortest paths from~$x$ to~$z$, one through each side of 
  $\pi_\eps(z)\setminus \{z\}$. 
  Let~$s$ be the \emph{last} common vertex on these
  two paths, and let~$C$ be the cycle formed by the two paths
  between~$s$ and~$z$.  We have~$d_G(v_1, z) = \lambda + d_G(x, s) + \len(C)/2
  \leq D$, implying that $\lambda + \len(C)/2 \leq D$. 

  We obtain a graph~$G'$ from~$G$ by applying
  Lemma~\ref{lem:identi-cycle} with the cycle~$C$ and also removing
  the leaf~$v_2$.  We have $\cyc(G') = k - 1$, $\leaves(G') \leq
  \leaves(G)$, and $\len(G') = \len(G) - {\len(C)}/{2} - \lambda$.  By the
  inductive assumption and since $\leaves(G) \geq 2$, we have
  \[
  \len(G') \leq \Big( \cyc(G') + \max\big\{1, \frac{\leaves(G')}{2}\big\}\Big) \cdot \diam(G')
  \leq \big( k - 1 + \frac{\leaves(G)}{2} \big) \cdot \diam(G')
  \]
  Using $\diam(G') \leq D$ and $\lambda + \len(C)/2 \leq D$ we now have
  \[
  \len(G) = \len(G') + \frac{\len(C)}{2} + \lambda
  \leq \len(G') + D
  \leq \big( k - 1 + \frac{\leaves(G)}{2} \big) \cdot D + D
  = \big( k + \frac{\leaves(G)}{2} \big) \cdot D,
  \]
  completing the proof.
\end{proof}

\begin{figure}[t]
  \centerline{\includegraphics[width=.9\textwidth]{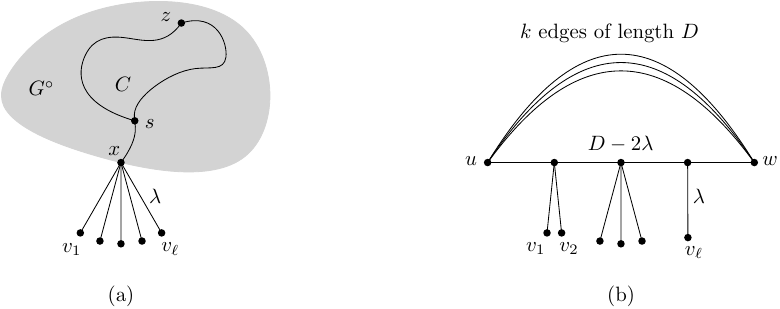}}
  \caption{(a) Proof of Theorem~\ref{thm:len-diam}. (b)
    $\lambda$-Medusa graph with diameter $D$, cyclomatic number $k$,
    and $\ell$ leaves.}
  \label{fig:lambda-medusa}
\end{figure}

We will now study the family of graphs for which equality holds in
Theorem~\ref{thm:len-diam}.  A \emph{$\lambda$-Medusa graph} with
diameter~$D$, cyclomatic number~$k$, and~$\ell$ leaves is constructed
as follows: Let $uw$ be an edge of length~$D - 2\lambda$. Create~$k$
(parallel) edges of length~$D$ connecting~$u$ and~$w$, and, finally,
add~$\ell$ leaves by adding edges of length~$\lambda$ with one
endpoint located on the edge~$uw$. See Figure~\ref{fig:lambda-medusa}(b).

It is easy to see that for $0 \leq \lambda \leq D/2$, a
$\lambda$-Medusa graph~$M_{\lambda}$ has diameter at most~$D$.
For~$\lambda = 0$, $M_0$~has no leaves, and simply consists of $k+1$
parallel edges of length~$D$, so $\len(M_0) = (k+1)D$.  For~$\lambda =
D/2$, the edge~$uw$ degenerates to a point and all leaves are attached
to this point---in other words, all leaves have a common neighbor and
have length~$D/2$, so we have~$\len(M_{D/2}) = (k + \ell/2) D$.  For
$0 < \lambda < D/2$ and~$\ell = 2$, the $\lambda$-Medusa
graph~$M_{\lambda}$ has again length~$\len(M_\lambda) = (k + 1) D$.
The $D/2$-Medusa graph for $k = 0$ is a star with
$\ell$~edges of equal length~$D/2$, the $0$-Medusa graph for~$k=1$ is a cycle of
length~$2D$.

We will need a small observation:
\begin{lemma}
  \label{lem:long-edges}
  In a metric graph~$G$ of diameter~$D$, two edges of length~$D$
  must share both endpoints.
\end{lemma}
\begin{proof}
  If the two edges do not share any endpoint, then their midpoints
  have distance larger than~$D$, a contradiction. Assume now that the
  edges are~$uv$ and~$uw$ with shared endpoint~$u$, and let~$\delta =
  d_G(v,w) > 0$. Let~$x$ and~$y$ be points on~$uv$ and~$uw$ at
  distance~$D/2 + \delta/3$ from~$u$.  Then $d_G(x, y) > D$. 
\end{proof}

\begin{theorem}
  \label{thm:medusa}
  The metric graphs~$G$ with
  \[
  \len(G) = \Big(\cyc(G) + \max\big\{1,
  \frac{\leaves(G)}{2}\big\}\Big) \cdot \diam(G)
  \]
  are exactly the following:
  \begin{itemize}
  \item the $0$-Medusa graphs (which have no leaves),
  \item the $\lambda$-Medusa graphs for $0 < \lambda < D/2$ with two
    leaves, and
  \item the $D/2$-Medusa graphs with two or more leaves.
  \end{itemize}
\end{theorem}

\begin{proof}
  It is easy to see that for all the listed graphs the equality holds.
  Thus, we focus on the other direction: if equality holds for~$G$,
  then $G$~is one of the listed graphs.
  
  Again, we proceed by induction over~$\cyc(G)$.  The case~$\cyc(G) =
  0$ follows from our discussion after Lemma~\ref{lem:len-diam-tree}.
  We therefore assume that we are given a
  graph~$G$ with~$k := \cyc(G) > 0$ that satisfies
  \[
  \len(G) = \big(k + \max\{1, \leaves(G)/2\}\big) \cdot D,
  \]
  where~$D := \diam(G)$, and that the claim holds for all graphs with
  cyclomatic number smaller than~$k$.
  
  Assume first that~$\leaves(G) = 1$, let~$v$ be this single
  leaf, and let~$n(v)$ be the neighbor vertex of $v$. 
  We can attach a second leaf of length~$\eps > 0$
  to~$n(v)$, obtaining a new graph of length greater than $(k +1) \cdot D$, 
  a contradiction to Theorem~\ref{thm:len-diam}. 
  Hence, $G$ cannot have exactly one leaf.

  We next assume that~$G$ has no leaves, so that $\len(G) = (k+1)D$.
  Let~$C$ be a shortest cycle in~$G$, which implies that $\len(C) \leq
  2D$.  If~$G = C$, then~$G$ is the $0$-Medusa graph with~$k = 1$.
  Otherwise, we consider the graph~$G'$ obtained from~$G$ by applying
  Lemma~\ref{lem:identi-cycle} with the cycle~$C$. We have
  $\cyc(G') = \cyc(G) - 1$ and
  \[
  \len(G') = \len(G) - \len(C)/2 = (k+1) D - \len(C)/2.
  \]
  By Theorem~\ref{thm:len-diam} we have $\len(G') \leq k D$,
  which implies~$\len(C) = 2D$ and $\len(G') = kD$. 
  Since $\leaves(G') \leq 1$, by the inductive assumption $G'$ is the $0$-Medusa graph 
  and actually has no leaves.  In fact, $G'$ consists of~$k$ parallel
  edges~$e_1, e_2, \dots, e_k$ of length~$D$ connecting two
  vertices~$u$ and~$w$.  To reconstruct~$G$, we need to undo the
  isometric path identification, that is, we replace a path~$\sigma$
  of length~$D$ by two paths of length~$D$.  If~$\sigma$ is identical to one of the
  edges~$e_i$, then~$G$ is again the $0$-Medusa graph, and we are
  done.  Otherwise, we can break symmetry and assume that the endpoints 
  of~$\sigma$ are a point~$x \in e_1$ and a point~$y \in
  e_2$, and that the vertex~$w$ of~$G'$ lies on~$\sigma$. There are
  then two vertices~$w_1, w_2 \in G$ with~$\pi(w_1) = \pi(w_2) = w$.
  (Recall that~$\pi$ is the projection from~$G$ to $G'$.)
  See Figure~\ref{fig:no-leaf-medusa}.
  The~$k-1$ edges~$e_2, e_3, \dots, e_k$ of~$G'$ are the images of
  edges that either connect~$u$ and~$w_1$ or~$u$ and~$w_2$ in~$G$. By
  Lemma~\ref{lem:long-edges}, we cannot have both types of edges, so
  we can assume that~$e_2,\dots,e_k$ connect~$u$ and~$w_1$. But
  then~$e_2$ and the edge $xy$ containing~$w_2$ contradict
  Lemma~\ref{lem:long-edges}, so this case cannot happen. 
  \begin{figure}[t]
    \centerline{\includegraphics[width=.4\textwidth]{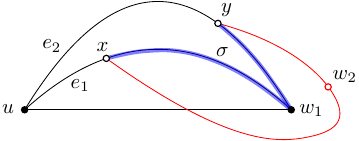}}
    \caption{Graph $G$ reconstructed from $G'$ when a path $\sigma$ has endpoints $x\in e_1$ and $y\in e_2$.}
    \label{fig:no-leaf-medusa}
  \end{figure}

  In the final case, we assume that~$G$ has $\ell \geq 2$ leaves and
  that $\len(G) = (k + \ell/2)D$.  If~$\ell \geq 3$, then all leaves
  must have length~$D/2$, as removing a leaf of shorter length would
  create a graph violating Theorem~\ref{thm:len-diam}.  This in turn
  implies that all leaves have the same neighbor~$x$.  We can remove
  all but two of the leaves and only need to characterize the
  case~$\ell = 2$. We can add the remaining $\ell-2$ leaves back to
  the graph to extend the characterization to the case~$\ell \geq 3$.

  In other words, we have to understand the case where $G$~has two
  leaves and~$\len(G) = (k+1)D$.  As in Lemma~\ref{lem:leaf-equal},
  let~$v_1$ and~$v_2$ be the two leaves with~$\lambda(v_1) \geq
  \lambda(v_2)$.  We can remove~$v_2$ from its neighbor and reattach
  it at~$x := n(v_1)$, the neighbor of $v_1$.  Assume for a
  contradiction that $\lambda(v_1) > \lambda(v_2)$. If $\lambda(v_1)
  \leq D/2$, then we can change the length of~$v_2$ to
  be~$\lambda(v_1)$, increasing the length of the graph in violation
  of Theorem~\ref{thm:len-diam}.  On the other hand, if~$\lambda(v_1)
  > D/2$, then we can reattach~$v_2$ at the point~$y$ on the
  edge~$v_1x$ with~$d_G(v_1, y) = D/2$ and increase its length
  to~$D/2$.  This again contradicts Theorem~\ref{thm:len-diam}.  It
  follows that $\lambda(v_1) = \lambda(v_2)$.

  We set~$\lambda := \lambda(v_1) = \lambda(v_2)$, and perform the
  same construction as in Theorem~\ref{thm:len-diam}: we identify a
  furthest point~$z \in G^{\circ}$ from~$x$, two shortest paths
  from~$x$ to~$z$, a last common point~$s$, a cycle~$C$ with antipodal
  points~$s$ and~$z$, and finally construct a graph~$G'$ by using
  Lemma~\ref{lem:identi-cycle} with the cycle~$C$, and additionally
  removing leaf~$v_2$.  More precisely, we identify the two half-cycles 
  of~$C$ between~$s$ and~$z$---this is possible because~$s$ has degree 
  larger than two, and will ensure that the only leaves of~$G'$ are~$v_1$
  and possibly~$z$.

  We have~$\len(G') = \len(G) - \len(C)/2 - \lambda = (k+1)D - \len(C)/2 -
  \lambda$, and since~$\cyc(G') = k-1$ and~$G'$ has at most two
  leaves, Theorem~\ref{thm:len-diam} implies that~$\len(G') \leq kD$,
  which gives us~$\len(C)/2 + \lambda \geq D$. Since~$d_G(v_1, z) =
  \lambda + \len(xs) + \len(C)/2 \leq D$, this tells us that~$\len(C) =
  2D-2\lambda$ and that~$x = s$.  So~$\len(G') = kD$, and by the
  inductive assumption,~$G'$ is a Medusa graph. Since~$G'$ has at least
  the leaf~$v_1$, it must actually have two leaves, so it is a
  $\lambda'$-Medusa graph with~$0 < \lambda' \leq D/2$.  The second
  leaf of~$G'$ is necessarily the point~$z$ of~$G$, so we
  have~$d_{G'}(v_1, z) = D$. This means that the two leaves are
  attached at the endpoints~$u$ and~$w$ of the base edge of the
  $\lambda'$-Medusa graph.

  If $k=1$, then~$G'$ is simply a path of length~$D$. Undoing the path
  identification, we immediately find that~$G$ is a $\lambda$-Medusa
  graph, so assume that~$k > 1$ and therefore~$\cyc(G') = k - 1 > 0$.

  We observe that the single edge~$v_1x$ of~$G$ is projected on a path
  of length~$\lambda$ in~$G'$ that cannot contain the vertex~$u$ of~$G'$ of degree
  three or higher, so we must have~$\lambda' \geq \lambda$.  Let~$\mu
  := \lambda' - \lambda$.

  We consider first the case~$\mu = 0$, so~$G'$ is a~$\lambda$-Medusa
  graph. Its base edge~$uw$ has length~$D-2\lambda$, there are~$k-1
  \geq 1$ parallel edges~$e_1, e_2, \dots, e_{k-1}$ of length~$D$
  connecting~$u$ and~$w$, the leaf~$v_1$ is attached to~$u$ by an edge
  of length~$\lambda$, and the leaf~$z$ is attached to~$w$ by an edge
  of length~$\lambda$.  Undoing the construction of~$G'$, we conclude
  that~$G$ has two edges~$uw_1$ and~$uw_2$ of length~$D-2\lambda$,
  $v_1$ and~$v_2$ are attached to~$u$, each by an edge of length~$\lambda$,
  and~$z$ is attached to both~$w_1$ and~$w_2$, each by an edge of
  length~$\lambda$. The other edges are of length~$D$ and connect~$u$
  with either~$w_1$ or~$w_2$. See Figure~\ref{fig:two-leaf-1}.
  \begin{figure}[t]
    \centerline{\includegraphics[width=.9\textwidth]{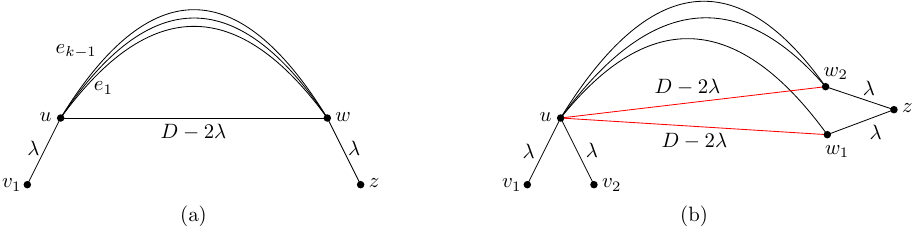}}
    \caption{(a) $G'$ and (b) its reconstruction $G$.}
    \label{fig:two-leaf-1}
  \end{figure}
  By Lemma~\ref{lem:long-edges}, we cannot have both types of edges,
  so we can assume that~$e_1,\dots,e_{k-1}$ connect~$u$ and~$w_1$.
  Then~$w_2$ and~$z$ are vertices of degree two, so in fact the
  path~$u w_2 z w_1$ is just another edge~$e_k$ of length~$D$ from~$u$
  to~$w_1$, and~$G$ is a $\lambda$-Medusa graph with base edge~$u
  w_1$.  (Note that when~$\lambda = D/2$, then~$u = w_1 = w_2$: this
  simplifies the argument, but the result is the same.)

  We turn to the case~$\mu > 0$, so~$G'$ is a~$(\lambda + \mu)$-Medusa
  graph.  Its base edge~$uw$ has length~$D - 2\lambda - 2\mu$, there
  are~$k-1 \geq 1$ parallel edges~$e_1, e_2, \dots, e_{k-1}$ of
  length~$D$ connecting~$u$ and~$w$, the leaf~$v_1$ is attached to~$u$
  by an edge of length~$\lambda + \mu$, and the leaf~$z$ is attached
  to~$w$ by an edge of length~$\lambda + \mu$.  We conclude that~$G$
  looks as follows: $G$ has two parallel paths $x u_1 w_1 z$ and~$x
  u_2 w_2 z$ of length~$D - \lambda$, where~$\len(xu_1) = \len(xu_2) = \mu$,
  $\len(u_1 w_1) = \len(u_2 w_2) = D - 2\lambda$, and~$\len(w_1 z) = \len(w_2 z) =
  \lambda + \mu$.  There are two leaves~$v_1, v_2$ attached to~$x$ by
  edges of length~$\lambda$.  The other edges are of length~$D$ and
  connect either~$u_1$ or~$u_2$ with either~$w_1$ or~$w_2$. See Figure~\ref{fig:two-leaf-2}.
  \begin{figure}[t]
    \centerline{\includegraphics[width=\textwidth]{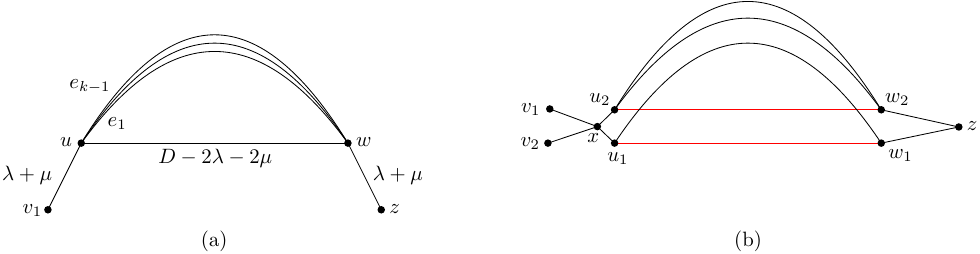}}
    \caption{(a) $G'$ and (b) its reconstruction $G$.}
    \label{fig:two-leaf-2}
  \end{figure}
  By Lemma~\ref{lem:long-edges}, we cannot have more than one type of
  these edges, so, breaking symmetry, we can assume
  that~$e_1,\dots,e_{k-1}$ connect either~$u_1$ and~$w_1$, or~$u_1$
  and~$w_2$.  Assume they connect~$u_1$ and~$w_1$.  Then the
  vertices~$w_2$ and $z$ are of degree two, so that the path~$u_2 w_2
  z w_1$ is in fact an edge of length~$D$ connecting~$u_2$ and~$w_1$,
  violating Lemma~\ref{lem:long-edges} with~$e_1$.  It follows
  that~$e_1, \dots, e_{k-1}$ connect~$u_1$ and~$w_2$.  Then~$w_1$
  and~$z$ are of degree two, so the path~$u_1 w_1 z w_2$ is in fact an
  edge~$e_k$ of length~$D$. On the other hand, the path~$u_1 x u_2
  w_2$ is of length~$\mu + \mu + (D - 2 \lambda - 2\mu) = D -
  2\lambda$, so our graph~$G$ is in fact a $\lambda$-Medusa graph with
  base edge~$u_1w_2$, with two leaves of length~$\lambda$ attached to
  the point~$x$ on this base edge. (Again note that when~$\lambda' =
  D/2$, then~$u_1 = w_1$ and~$u_2 = w_2$. It is still true that the
  edges~$e_1,\dots,e_{k-1}$ must connect~$u_1$ and~$u_2$, and the rest
  of the argument goes through unchanged.)

  In all cases we have now established that~$G$ is a $\lambda$-Medusa
  graph.  Recall that if~$G$ originally had at least three leaves,
  then all leaves have length~$D/2$, so~$\lambda = D/2$.  Then the
  base edge collapses to a point, all leaves are attached here.  On
  the other hand, if~$G$ originally had two leaves, then any leaf
  length~$0 < \lambda \leq D/2$ is possible.  If~$\lambda < D/2$, then
  we can move the second leaf back to its original position---any
  position on the base edge of the $\lambda$-Medusa graph is possible.
\end{proof}


\section{Applications}
\label{sec:applications}

Consider the metric graph~$G$ that consists of a single cycle of length~$L$. 
Clearly it has~$\diam(G) = L/2$.  Can we improve the diameter of~$G$ by
adding ``wormholes,'' that is, additional edges of length zero?

It is not difficult to see that adding a single wormhole cannot decrease 
the diameter of~$G$, but using two wormholes that connect antipodal points 
equally spaced on the cycle, we can improve the diameter to~$L/4$.
The question is: is this the best one can do with two wormholes?

Indeed it is, as we can see by considering the \emph{wormhole graph}~$G_w$, 
which we obtain from~$G$ by identifying the endpoints of each wormhole.
Clearly, $\len(G_w) = \len(G) = L$, $\leaves(G_w) = 0$, and $\cyc(G_w) = 3$, 
as each identification increases the cyclomatic number by one.
Theorem~\ref{thm:len-diam} now immediately gives us~$\diam(G_w) \geq L/4$.

The lower bound generalizes to any number of wormholes: For~$w$
wormholes, we have~$\len(G_w) \geq L / (w + 2)$.  For even~$w$, this
bound can indeed be achieved, by using $w/2$ wormholes each to
identify the corners of two regular~$(w/2+1)$-gons that are rotated by
$L/(w+2)$ with respect to each other. (For an odd number of wormholes,
we do not know what diameter is possible.  We conjecture that the
diameter achievable with~$2n+1$ wormholes is the same as for~$2n$
wormholes.)

What bounds are possible if we require the wormholes to be connected?
In that case, the $w$ wormholes serve to identify~$w+1$ points in~$G$,
which is the fusion of $w+1$ points.  The wormhole graph has a single
vertex with $w+1$ attached loops, whose lengths add up to~$L$.  This
implies that the best diameter that can be achieved this way
is~$L/(w+1)$. This implies that for even~$w$ vertex-fusion is strictly
less powerful than general wormholes.

We can ask the same question for general metric graphs~$G$: can we add
$w$~wormholes to~$G$ to reduce the diameter, and what diameter can be
achieved?  Theorem~\ref{thm:len-diam} gives us a lower bound on this
diameter, by observing that for the wormhole graph~$G_w$ it holds that
$\len(G_w) = \len(G)$, $\cyc(G_w) = \cyc(G) + w$, and $\leaves(G_w)
\leq \leaves(G)$ (it can be less when wormholes are placed at leaves).

When~$G$ is simply a path of length~$L$, the lower bound for~$w$
wormholes is~$\diam(G_w) \geq L / (w+1)$, and this diameter can indeed
be achieved by placing the wormholes such that~$G_w$ is a
$\lambda$-Medusa graph with two leaves, for $0 < \lambda \leq
L/(2w+2)$.  For~$\lambda = L/(2w+2)$ the solution corresponds to
vertex fusion, but for all other values of~$\lambda$ this is not the
case.

When~$G$ is a star consisting of~$s$ ``spikes'' of length one,
then~$\len(G) = s$ and $\diam(G) = 2$.  With $ns$ wormholes, for a
positive integer~$n$, the diameter can be decreased to $2/(2n+1)$; see
Figure~\ref{fig:star-optimal}. Note that in the displayed construction
the wormholes are connected, and thus it is fusion.  This is again
optimal as we can see by considering the wormhole graph~$G_w$:
$\len(G_w) = s$, $\leaves(G_w) = s$, $\cyc(G_w) = ns$, and so
$\diam(G_w) \geq 2/(2n+1)$.

\begin{figure}[t]
  \centerline{\includegraphics[]{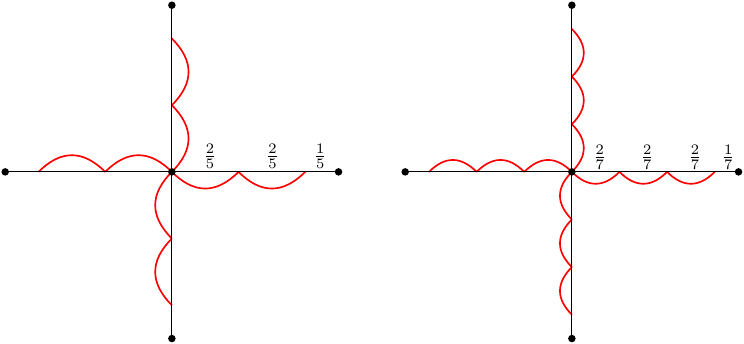}}
  \caption{Star (in black) with $s=4$ spikes of unit length and $ns$ 
    wormholes (in red) that reduce the diameter to $2/(2n+1)$. 
    On the left the case $n=2$, and in the right the case $n=3$}
  \label{fig:star-optimal}
\end{figure}

For the case when~$G$ is a star consisting of~$s$ spikes of length
one, and we can use at most $s-1$~wormholes, we have different ad-hoc
techniques to bound the diameter that can be achieved.

\begin{proposition}
  \label{pro:s-1-spikes}
  Let $G$ be a metric star with center~$c$ and~$s$ spikes~$cv_{i}$ of
  length one, for $i \in \{1, \dots, s\}$, and let $G_w$ be a graph
  obtained by adding $s-1$~wormholes to~$G$. Then for some~$v_{i}$ we
  have~$d_{G_{w}}(c, v_{i}) = 1$, implying~$\diam(G_w)\ge 1$. This
  bound is tight.
\end{proposition}
\begin{proof}
  By identifying~$s$ leaves~$v_{1}, \dots, v_{s}$ into a single point,
  which requires $s-1$~wormholes, we obtain a collection of $s$
  parallel edges of unit length. This graph has diameter~$1$ (it is
  actually the $0$-Medusa graph). This shows that the bound is tight.
	
  It remains to prove the lower bound. We prefer here to
  consider~$G_w$ as being obtained from $G$ by adding $s-1$
  ``wormhole'' edges of zero length.  Let~$T$ be a shortest-path tree
  with origin~$c$ for~$G_{w}$.  For each~$v_{i}$, consider the
  path~$\gamma_{i}$ from~$c$ to~$v_{i}$ in~$T$.  Let~$\gamma'_{i}$ be
  the part of~$\gamma_{i}$ from the last leaf preceding~$v_{i}$ up
  to~$v_{i}$ (if there is no other leaf on~$\gamma_{i}$,
  then~$\gamma'_{i} = \gamma_{i}$).  There are~$s$ disjoint
  paths~$\gamma'_{i}$ in~$T$, so one of them cannot contain a wormhole
  edge. For this~$\gamma'_{i}$ we therefore have~$\len(\gamma'_{i})
  \geq 1$, implying~$d_{G_{w}}(c, v_{i}) \geq 1$.
\end{proof}

\begin{proposition}
  \label{pro:s-2-spikes}
  Let $G$ be a metric star consisting of~$s$ spikes of length one and
  let $G_w$ be a graph obtained by adding~$s-2$ wormholes to~$G$.
  Then $\diam(G_w)=\diam(G)=2$.
\end{proposition}
\begin{proof}
  Let again $c$ be the center and let $v_1,\dots,v_s$ be the leaves
  of~$G$.  Consider the combinatorial multi-graph $H$ with vertex set
  $\{1,\dots,s\}$, with an edge between~$i$ and~$j$ for each wormhole
  identifying a point on~$c v_i$ with a point on~$c v_j$. If a
  wormhole identifies the center~$c$ with a point on~$c v_i$, we add
  to~$H$ a loop at~$i$.  The graph $H$ has~$s$ vertices and exactly
  $s-2$ edges; it may have parallel edges, loops, and parallel loops.
  It follows that $H$ has at least two connected components $H^1$ and
  $H^2$ such that $|E(H^i)|\le |V(H^i)|-1$ (for $i=1,2$). Note that
  such a component may consist of an isolated vertex.
	
  For $i=1,2$, let $G^i_w$ be the subgraph of~$G_w$ given by the
  spikes corresponding to~$V(H^i)$, and the wormholes between
  them. Since~$c$ is the only common point between~$G^i_w$ and the
  rest of~$G_w$, we have
  \begin{align}
    \forall x_1, x_{2}\in G^i_w: ~~~
    & d_{G_w}(x_1,x_2) = d_{G^{i}_w}(x_1,x_2) \\
    \forall x_1\in G^1_w,~ x_2\in G^2_w: ~~~
    & d_{G_w}(x_1,x_2) = d_{G_w}(x_1,c)+ d_{G_w}(c,x_2)
    = d_{G^1_w}(x_1,c)+ d_{G^2_w}(c,x_2).	
    \label{eq:s-2-wormholes}
  \end{align}
  Note that (for $i=1,2$) the graph $G^i_w$ is a wormhole graph
  obtained from a star with $s_i=|V(H_i)|$ spikes of unit length and
  $|E(H_i)|\le s_i-1$ wormholes.  By Proposition~\ref{pro:s-1-spikes}
  there is a leaf~$v_i$ such that $d_{G^i_w}(c,v_i)=1$. It follows
  from Equation~\eqref{eq:s-2-wormholes} that $d_{G_w}(v_1,v_2)=2$.
\end{proof}

\bibliography{paper}
\end{document}